\documentclass[12pt]{amsart}
\usepackage{amsmath,amssymb,amsthm}

\usepackage{color}

\newtheorem{thm}{Theorem}[section]
\newtheorem{pro}[thm]{Proposition}

\newtheorem{cor}[thm]{Corollary}
\newcommand{\Tee }{\mathcal T}

\newtheorem{dfn}{Definition}[section]

\newtheorem{qu}[thm]{Question}

\theoremstyle{definition}

\newtheorem{ex}[thm]{Example}

\theoremstyle{remark}

\begin{document}

\title[Spectral representations of topological groups and near-openly generated groups]
{Spectral representations of topological groups and near-openly generated groups}

\author{K. Kozlov}
\address{Department of General Topology and Geometry, Faculty of Mechanics and Mathematics, \\
Moscow State University, Moscow 119991, Russia}
\email{kkozlov@mech.math.msu.su}
\thanks{The first author was partially supported by RFBR, 17-51-18051 Bulg\_a}

\author{V. Valov}
\address{Department of Computer Science and Mathematics, Nipissing University,
100 College Drive, P.O. Box 5002, North Bay, ON, P1B 8L7, Canada}
\email{veskov@nipissingu.ca}
\thanks{The second author was partially supported by NSERC Grant 261914-13.}





\begin{abstract}
Near-openly generated groups are introduced. It is a topological and
multiplicative subclass of $\mathbb R$-factorizable groups. Dense
and open subgroups, quotients and Raikov completion of a near-openly
generated group are near-openly generated. Almost connected pro-Lie
groups, lindel\" off almost metrizable groups and the spaces
$C_p(X)$ of all continuous real-valued functions on a Tychonoff
space $X$ with pointwise convergence topology are near-openly
generated.

We provide characterizations of near-openly generated groups using
methods of inverse spectra and topological game theory.
\end{abstract}

\maketitle\markboth{}{Spectral representations of topological groups}

\section{Introduction}

Inverse spectra is a useful device in topology and topological
algebra. They provide a technique for approximating of complicated
spaces by simple ones. L.~Pontryagin constructed a Lie group series
for compact groups --- transfinite continuous spectra of groups
$G_\alpha$ and open homomorphisms such that the kernels of bonding
homomorphisms $p^{\alpha+1}_\alpha$ and $G_0$ are compact Lie
groups. Another examples of groups which are defined and
investigated by inverse spectra are pro-Lie groups \cite{hm} and
almost connected pro-Lie groups~\cite{hm1}. Some more applications
of inverse spectra in the category of groups can be found
in~\cite{ch}, \cite{K2013}, \cite{k2017}, \cite{TkL}, \cite{pa},
\cite{sk}, \cite{us}.

Everywhere below by a {\em spectrum} we mean an inverse system.
Spectrum $\displaystyle S=\{X_\alpha, p^{\beta}_\alpha, A\}$ is
called {\em almost continuous $\sigma$-spectrum} if it satisfies the
following conditions:
\begin{itemize}
\item[(1)] all spaces $X_\alpha$ are second-countable and the bonding maps $p^{\beta}_\alpha$ are
surjective;
\item[(2)] the directed set $A$ is $\sigma$-complete (every increasing sequence
in $A$ has a supremum in $A$);
\item[(3)] for every increasing sequence $\{\alpha_n\}\subset A$ the
space $X_\beta$ is a dense subset of
$\displaystyle\lim_\leftarrow\{X_{\alpha_n},p^{\alpha_{n+1}}_{\alpha_n},n\geq
1\}$, where $\beta=\sup\{\alpha_n\}$.
\end{itemize}
If a space $X$ is embedded in $\displaystyle\lim_\leftarrow S$ in
such a way that $p_\alpha(X)=X_\alpha$ for each $\alpha$, where
$p_\alpha\colon\displaystyle\lim_\leftarrow S\to X_\alpha$ is the
$\alpha$-th limit projection, then we say that $X$ is the {\em
almost limit} of spectrum $S$, notation
$X=\mathrm{a}-\displaystyle\lim_\leftarrow S$. In case $X$ is the
almost limit of a spectrum $S$ and for every continuous real-valued
function $f$ on $X$ there exists $\alpha\in A$ and a continuous
function $f_\alpha$ on $X_\alpha$ such that $f=f_\alpha\circ
p_\alpha$, then $S$ is said to be {\em a factorizing spectrum}.

Recall that a continuous map $f:X\to Y$ is {\em nearly open}
\cite{ArTk} (resp., {\em skeletal}) if
$f(U)\subset\mathrm{Int}\overline{f(U)}$ (resp.,
$\mathrm{Int}\overline{f(U)}\neq\varnothing$) for every open
$U\subset X$, where $\overline{f(U)}$ denotes the closure of $f(U)$
in $Y$. Nearly open maps were introduced in~\cite{pt} (see
also~\cite{tk}, where nearly open maps were called {\it $d$-open}).
In~\cite{va1} skeletal maps are called {\it $ad$-open}.

\begin{dfn}
A topological group $G$ is near-openly generated if
$G=\mathrm{a}-\displaystyle\lim_\leftarrow S_G$, where
$\displaystyle S_G=\{G_\alpha, p^{\beta}_\alpha, A\}$ is a
factorizing almost continuous $\sigma$-spectrum consisting of
second-countable topological groups $G_\alpha$ and continuous nearly
open homomorphisms $p^{\beta}_\alpha$.
\end{dfn}

The aim of the paper is to describe the class of near-openly
generated topological groups. This class is topological (i.e.
invariant under homeomorphisms) and has nice properties. It is
multiplicative. Dense and open subgroups, quotients and Raikov
completion of a near-openly generated group are near-openly
generated. Almost connected pro-Lie groups (in particular compact
groups) and lindel\" off almost metrizable groups~\cite{pa} are
near-openly generated.  Another example of such groups are the
spaces $C_p(X)$ of all continuous real-valued functions on a
Tychonoff space $X$ with pointwise convergence topology. It is worth
to note that the class of near-openly generated groups is a subclass
of $\mathbb R$-factorizable groups for which the problems whether it
is multiplicative and topological are unsolved.

Next theorem provides a topological characterization of near-openly
generated groups.

{\rm THEOREM~\ref{Main}.} The following are equivalent for a
topological group $G$:
\begin{itemize}
\item[(1)] $G$ is near-openly generated;
\item[(2)] $G$ is $\mathrm{I}$-favorable;
\item[(3)] $G$ has a $\sigma$-lattice of skeletal maps;
\item[(4)] $G$ has a $\sigma$-lattice of nearly open homomorphisms.
\end{itemize}

\medskip

$\mathrm{I}$-favorable spaces \cite{dkz} and $\sigma$-lattices of
maps on a given space are defined in \S\ 2. Note that, by
Theorem~\ref{local}, $\mathrm{I}$-favorability of an $\omega$-narrow
group is a local property.

It follows from the results of \S\ 2 that near-openly generated
groups can be described as dense subgroups of the limits of almost
continuous $\sigma$-spectra of topological groups and continuous
nearly open homomorphisms (not necessarily factorizing). At the same
time, near-openly generated groups are exactly the groups whose
underlying spaces are almost limits of almost continuous
$\sigma$-spectra with skeletal bonding maps (such spaces are called
skeletally generated spaces~\cite{v}).

We also establish characterization of near-openly generated groups
using their isomorphic embeddings in products of second countable
topological groups (see \S\ 2 for the definitions of $\pi$-regular
and regular embeddings).

{\rm THEOREM~\ref{Main1}.} The following are equivalent for a
topological group $G$:
\begin{itemize}
\item[(1)] $G$ is near-openly generated;
\item[(2)] $G$ is topologically isomorphic to a subgroup of a product of second countable topological groups and any such an embedding is
$\pi$-regular;
\item[(3)] $G$ is topologically isomorphic to a subgroup of a product of second countable topological groups and any such an embedding is
regular.
\end{itemize}

\medskip

The paper is organized as follows. \S\ 2 contains a preliminary
information about $\mathrm{I}$-favorable and skeletally generated
spaces. \S\ 3 contains the proofs of Theorem~\ref{Main},
Theorem~\ref{Main1} and Theorem~\ref{local} mentioned above. We also
provide the main properties of near-openly generated groups and
examples of such groups. The last \S\ 4 contains necessary
information about spectra and $\sigma$-lattices.
Proposition~\ref{lemmain} is a spectral theorem of the form
spectrum--lattice. The idea is similar to the classical spectral
theorem for factorizing spectra~\cite{Chi} and the theorem on
intersection of lattices~\cite{sc1}. We also provide a spectral
representation of $\mathbb R$-factorizable groups~\cite{tk91}.

All spaces are assumed to be Tychonoff and the maps are continuous.
For information about topological groups and spectra see~\cite{ArTk}
and~\cite{Chi}, respectively.

\section{$\mathrm I$-favorable and skeletally generated spaces}

$\mathrm I$-favorable spaces were introduced in~\cite{dkz}. Two
players are playing the so called \textit{open-open game} in a space
$(X,\Tee_X)$. The players take countably many turns, a round
consists of player I choosing a nonempty open set $U\subset X$ and
player II a nonempty open set $V\subset U$; player I wins if the
union of II's open sets is dense in $X$, otherwise player II wins. A
space $X$ is called {\em $\mathrm I$-favorable} if player I has a
winning strategy. This means that there exists a function
$\mu:\bigcup_{n\geq 0}\Tee_X^n\to\Tee_X$ such that for each game
$$\mu(\varnothing),B_0,\mu(B_0),B_1,\mu(B_0,B_1),B_2,\ldots, B_n,\mu(B_0,\ldots,B_n),B_{n+1},\ldots$$
the union  $\bigcup_{n\geq 0}B_n$ is dense in $X$, where
$\mu(\varnothing)\neq\varnothing$ and
$B_{k+1}\subset\mu(B_0,B_1,..,B_k)\neq\varnothing$ and
$\varnothing\neq B_k\in\Tee_X$ for all $k\geq 0$.

If there exists an almost continuous $\sigma$-spectrum $S$ with
skeletal  bonding maps such that
$X=\mathrm{a}-\displaystyle\lim_\leftarrow S$, we say that {\em $X$
is skeletally generated}, see~\cite{v}. Let us mention that
in~\cite{vv} a space $X$ is called skeletally generated if $X$ is an
almost limit of a factorizing almost continuous $\sigma$-spectrum.
It follows from Proposition~\ref{lattice} below that the two
definitions are equivalent.

Recall that a subspace $X$ of a space $Y$ is {\em $\pi$-regularly
embedded in $Y$}~\cite{va1} if there exists a function ({\it a
$\pi$-regular operator}) $\mathrm{e}\colon\Tee_X\to\Tee_Y$ between
the topologies of $X$ and $Y$ such that:
\begin{itemize}
\item[$(\mathrm{e}1)$] $\mathrm{e}(\varnothing)=\varnothing$ and
$\mathrm{e}(U)\cap X$ is a dense subset of $U$;
\item[$(\mathrm{e}2)$] $\mathrm{e}(U)\cap\mathrm{e}(V)=\varnothing$ for any $U,V\in\Tee_X$
provided $U\cap V=\varnothing$.
\end{itemize}
A similar function with the additional property
$\mathrm{e}(U)\cap\mathrm{e}(V)=\mathrm{e}(U\cap V)$ for all
$U,V\in\Tee_X$ was also called a  {\it $\pi$-regular extension
operator} in~\cite{sh}). If $\mathrm{e}$ satisfies condition
$(\mathrm{e}2)$ and $\mathrm{e}(U)\cap X=U$ for all $U\in\Tee_X$,
the operator $\mathrm{e}$ is called {\em regular}, see~\cite{shi}.

Following~\cite{sc1}, \cite[definition 8]{s} and~\cite{va1}, we say
that a family $\Psi$ of maps from a space $X$ to second-countable
spaces is a {\em $\sigma$-lattice on $X$}  if $\Psi$ satisfies the
following conditions:
\begin{itemize}
\item[(L1)] If $\{\varphi_n\}\subset\Psi$ such that $\varphi_{n+1}\prec\varphi_n$ for all $n$, then $\Psi$ contains the diagonal product
$\varphi=\triangle_{n\geq 1}\varphi_n$ (here
$\varphi_{n+1}\prec\varphi_n$ means that there is a continuous map
$\varphi^{n+1}_n:\varphi_{n+1}(X)\to\varphi_{n}(X)$ such that
$\varphi_{n}=\varphi^{n+1}_n\circ\varphi_{n+1}$).
\item[(L2)] For any map $f:X\to f(X)$ with $f(X)$ having a countable weight there is $\varphi\in\Psi$ such that $\varphi\prec f$ (i.e., there is a map
$h:\varphi(X)\to f(X)$ such that $f=h\circ\varphi$).
\end{itemize}

The following characterizations  of skeletally generated spaces
(using the present definition) were established in~\cite[theorem
1.1]{v}.

\begin{thm}\cite{v}\label{Ifav}
For a space $X$ the following are equivalent:
\begin{itemize}
\item[{\rm (1)}] $X$ is $\mathrm{I}$-favorable;
\item[{\rm (2)}] every embedding of $X$ in another space is $\pi$-regular;
\item[{\rm (3)}] $X$ is skeletally generated.
\end{itemize}
\end{thm}

We can extend the list of conditions characterizing skeletally
generated spaces. A family of sets $\{B_t\}_{t<\lambda}$ indexed by
ordinals such that $B_{t}\subset B_{t+1}$, $t+1<\lambda$, and
$B_\alpha=\bigcup\{B_t: t<\alpha\}$ for limit ordinal
$\alpha<\lambda$ is called an {\it increasing transfinite family of
sets}.

\begin{pro}\label{lattice}
For a space $X$ the following are equivalent:
\begin{itemize}
\item[{\rm (1)}] $X$ is skeletally generated;
\item[{\rm (2)}] $X$ has a $\sigma$-lattice of skeletal maps;
\item[{\rm (3)}] there exists a factorizing almost continuous $\sigma$-spectrum $S$ with
skeletal  bonding maps such that
$X=\mathrm{a}-\displaystyle\lim_\leftarrow S$.
\end{itemize}
\end{pro}

\begin{proof}
$(1)\Rightarrow (2)$. Suppose that $X$ is skeletally generated.
Then, since $X$ has countable cellularity, since it is
$\mathrm{I}$-favorable~\cite[theoerem 1.1]{dkz}. If $X$ is
metrizable, then $X$ is second-countable and the identity map forms
a $\sigma$-lattice with the required properties. So, let $X$ be of
uncountable weight. We consider $X$ as a $C$-embedded subset of a
product $\Pi=\prod_{\alpha\in A}X_\alpha$ of second-countable
spaces. Then there exists a $\pi$-regular operator $\mathrm
e:\mathcal T_X\to\mathcal T_\Pi$.

For any set $B\subset A$ fix a standard open base $\mathcal B_B$ for
$\prod_{\alpha\in B}X_\alpha$ of cardinality $\max\{|B|,
\aleph_0\}$, and consider the projection
$\pi_B:\Pi\to\prod_{\alpha\in B}X_\alpha$. We say that a set
$B\subset A$ is {\em $\mathrm{e}$-admissible} if
$$\pi_B^{-1}(\pi_B(\overline{\mathrm{e}(\pi_B^{-1}(U)\cap
X)}))=\overline{\mathrm{e}(\pi_B^{-1}(U)\cap X)}$$ for all
$U\in\mathcal B_B.$ Because for every open set $W\subset\Pi$ there
is a countable set $B_W\subset A$ with
$\pi_{B_W}^{-1}(\pi_{B_W}(\overline{W}))=\overline{W}$, one can show
that the family $\mathcal A$ of all $\mathrm{e}$-admissible subsets
of $A$ has the following properties, where $X_B=\pi_B(X)$ (see the
proof of~\cite[proposition 3.1(ii)]{kpv1} or~\cite[proposition
3.7]{vv}:
\begin{itemize}
\item for every open $V\subset X$ and every $B\in\mathcal A$ we have $\pi_B(\mathrm{e}(V))\subset\overline{\pi_B(V)}$
and the restriction map $p_B=\pi_B|_X:X\to X_B$ is skeletal;
\item the union of every increasing transfinite family $\{B_t\}\subset\mathcal A$ belongs to $\mathcal A$;
\item for any set $C\subset A$ there is $B_C\in\mathcal A$ of cardinality $|B_C|=\max\{|C|, \aleph_0\}$ with $C\subset B_C$.
\end{itemize}

\smallskip

\noindent{\it  Claim 1.} For every surjective map $f:X\to M$, where
$M$ is a second-countable space, there is a countable
$\rm{e}$-admissible set $B\subset A$ such that $p_B\prec f$.

\smallskip

Indeed, let $\{h_k\}$ be a sequence of continuous functions on $M$
determining the topology of $M$. Since $X$ is $C$-embedded in $\Pi$,
each $h_k\circ f$ can be continuously extended to a function $f_k$
on $\Pi$. So, there is a countable set $B\subset A$ such that all
$f_k$ can be factored through $\Pi_B=\prod_{\alpha\in B}X_\alpha$.
Consequently, for any $k$ there is a function $g_k$ on $\Pi_B$ with
$g_k\circ\pi_B=f_k$. The maps $g_k$ determine a map
$g:X_f=\pi_B(X)\to M$ such that $f=g\circ\pi_B|_X$. Because every
countable subset of $B$ is contained in a countable $\rm
e$-admissible set, we may assume that $B$ is also $\rm
e$-admissible.

\smallskip

Let us show that the family $\Psi$ of all maps $p_B$, $B\subset A$
is countable and $\rm e$-admissible, form a $\sigma$-lattice of
skeletal maps for $X$. Note that every $p_B\in\Psi$ is skeletal. By
claim 1, $\Psi$ satisfies condition $(L2)$. To show that $\Psi$
satisfies also $(L1)$, suppose $\{B_n\}$ is a sequence of countable
$\rm e$-admissible subsets of $A$ such that $p_{B_{n+1}}\prec
p_{B_n}$ for all $n$. Note that the last relation does not imply
that each $B_n$ is a subset of $B_{n+1}$, but we have surjective
maps $q^{n+k}_n: X_{{B_{n+k}}}\to X_{B_n}$, $n,k\geq 1$, such that
$q^n_{n-1}\circ q^{n+1}_n=q^{n+1}_{n-1}$ for each $n$. So, there is
an inverse sequence $S=\{X_{{B_{n}}}, q^{n+1}_n, n\geq 1\}$. Because
$p_B\prec\triangle_np_{B_n}\prec p_B$, where $B=\bigcup_{n=1}^\infty
B_n$, $X_B$ is a dense set in $\displaystyle\lim_\leftarrow S$ and
there are maps $q_n: X_B\to X_{B_n}$ with $q_n\circ p_B=p_{B_n}$,
$n\geq 1$. This means that if $\mathcal B_n$ is a standard open base
for $\Pi_{B_n}=\prod_{\alpha\in B_n}X_\alpha$, then the open family
$\mathcal B=\{q_n^{-1}(V\cap X_{B_n}): V\in\mathcal B_n, n\geq 1\}$
is a base for $X_B$. It is clear that each $W\in\mathcal B$ is the
intersection of $X_B$ with an element of the standard open base for
$\Pi_B$. Because each $B_n$ is $\rm e$-admissible, we have
$$\pi_{B_n}^{-1}(\pi_{B_n}(\overline{\rm e(p_{B_n}^{-1}(V\cap
X_{B_n}))}))=\overline{\rm e(p_{B_n}^{-1}(V\cap X_{B_n}))}.$$ Since
$\pi_B\prec\pi_{B_n}$ for all $n$, we obtain that if
$W=q_n^{-1}(V\cap X_{B_n})$, then $$\rm e(p_B^{-1}(W))=\rm
e(p_{B_n}^{-1}(V\cap X_{B_n}))\ \mbox{and}\
\pi_B^{-1}(\pi_{B}(\overline{\rm e(p_{B}^{-1}(W))}))=\overline{\rm
e(p_{B}^{-1}(W))}.$$ The last equalities imply that $p_B$ is $\rm
e$-admissible.

$(2)\Rightarrow (3)$ follows from proposition~\ref{sp-lat}.

$(3)\Rightarrow (1)$. This implication is trivial.
\end{proof}

\begin{cor}\cite{v}\label{Icor}
A space $X$ is $\mathrm{I}$-favorable if and only if $X$ is
$\pi$-regularly embedded in a product of second-countable spaces.
\end{cor}

The first item of next proposition is from~\cite{v}, while other
items were established in~\cite[Corollary 4]{dkz}.

\begin{pro}\cite{dkz},\cite{v}\label{proIfav}
Let $X$ be an $\mathrm{I}$-favorable space. Then:
\begin{itemize}
\item every open subset of $X$ is $\mathrm{I}$-favorable;
\item every dense subspace of $X$ is $\mathrm{I}$-favorable;
\item every space containing $X$ as a dense subspace is $\mathrm{I}$-favorable;
\item every image of $X$ under a skeletal map is
$\mathrm{I}$-favorable;
\item a product of $\mathrm{I}$-favorable spaces is $\mathrm{I}$-favorable.
\end{itemize}
\end{pro}

A subspace $Y$ of $X$ is said to be {\it $z$-embedded} in $X$ if for
every zero-set $F$ in $Y$ there exists a zero-set $\Phi$ in $X$ such
that $F=Y\cap\Phi$.

\begin{pro}\label{pi-z}
If $Y$ is $\pi$-regularly embedded in the product $\Pi$ of
second-countable spaces, then $Y$ is $z$-embedded in $\Pi$.
\end{pro}

\begin{proof}  Let $\rm e:\Tee_{Y}\to\Tee_{\Pi}$ be a $\pi$-regular
operator, and $F\in Y$ be a zero-set. We represent $F$ as the
intersection of a decreasing sequence $\{U_i\}$ of open sets
$U_i\subset Y$ with $\overline U_{i+1}\subset U_i$ for each $i\geq
1$. Since $\Pi$ is a product of second-countable spaces, all
$\overline{\rm e(U_i)}$ are zero-sets in $\Pi$. Obviously,
$F\subset\Phi=\bigcap\overline{\rm e(U_i)}$. If there is a point
$x\in (\Phi\cap Y)\setminus F$, then $x\not\in\overline U_k$ for
some $k$. So, we can find a neighborhood $V\subset Y$ of $x$ with
$V\cap\overline U_k=\varnothing$. Consequently, $x\in\overline{\rm
e(V)}$ and $\rm e(V)\cap\rm e(U_k)=\varnothing$, which contradicts
$x\in\overline{\rm e(U_k)}$. Hence, $F=\Phi\cap X$.
\end{proof}

We also need next proposition established in~\cite[proposition
3.3]{v}. Recall that a transfinite spectrum $S=\{X_\alpha,
p^{\beta}_\alpha, \alpha\leq\beta<\tau\}$ is almost continuous if
for any limit ordinal $\beta$ the space $X_\beta$ is a (dense)
subset of
$\displaystyle\lim_\leftarrow\{X_{\alpha},p^{\alpha'}_{\alpha},
\alpha\leq\alpha'<\beta\}$.

\begin{pro}\cite{v}\label{regemb}
Let $S=\{X_\alpha, p^{\beta}_\alpha, \alpha\leq\beta<\tau\}$ be a
transfinite almost continuous spectrum with nearly open bonding maps
such that $X=\displaystyle\mathrm{a}-\underleftarrow{\lim}S$. Then:
\begin{itemize}
\item[{\rm (1)}] $X$ is regularly embedded in $\prod_{\alpha<\tau}X_\alpha$;
\item[{\rm (2)}] if, additionally, each $X_\alpha$ is regularly embedded in a space $Y_\alpha$, then $X$ is regularly embedded in $\prod_{\alpha<\tau}Y_\alpha$.
\end{itemize}
\end{pro}

\section{Near-openly generated topological groups}

\subsection{Characterizations and properties}

A topological group homeomorphic to an $\mathrm{I}$-favorable space
is called {\em $\mathrm{I}$-favorable group}. Following~\cite{ArTk},
a topological group $G$ is said to be {\em $\mathbb R$-factorizable}
if for every continuous real valued function $f$ on $G$ there is a
continuous homomorphism $\pi:G\to K$ onto a second-countable
topological group $K$ and a continuous function $h$ on $K$ such that
$f=h\circ\pi$. We also say that $G$ is an {\em $\omega$-narrow
group~\cite{ArTk}} if $G$ is topologically isomorphic to a subgroup
of a product of second-countable groups.

\begin{pro}\label{IfavRfact}
Any $\mathrm{I}$-favorable topological group is $\mathbb
R$-factorizable.
\end{pro}

\begin{proof}
Any $\mathrm{I}$-favorable space has countable
cellularity~\cite{dkz}. Therefore, $G$ is an $\omega$-narrow
group~\cite[theorem 3.4.7]{ArTk}. So, $G$ is topologically
isomorphic to a subgroup of a product $\Pi$ of second-countable
groups. By theorem~\ref{Ifav} the embedding of $G$ in $\Pi$ is
$\pi$-regular and, by proposition~\ref{pi-z}, it is a $z$-embedding
into the $\mathbb R$-factorizable group $\Pi$. Therefore, according
to~\cite[theorem 8.2.6]{ArTk}, $G$ is an $\mathbb R$-factorizable
group.
\end{proof}

\begin{thm}\label{Main}
The following are equivalent for a topological group $G$:
\begin{itemize}
\item[{\rm (1)}] $G$ is near-openly generated;
\item[{\rm (2)}] $G$ is $\mathrm{I}$-favorable;
\item[{\rm (3)}] $G$ has a $\sigma$-lattice of skeletal maps;
\item[{\rm (4)}] $G$ has a $\sigma$-lattice of nearly open homomorphisms.
\end{itemize}
\end{thm}

\begin{proof}
$(1)\Rightarrow (2)$. The space of a near-openly generated group is
skeletally generated. So, by theorem~\ref{Ifav}, $G$ is
$\mathrm{I}$-favorable.

$(2)\Rightarrow (3)$. Suppose $G$ is $\mathrm{I}$-favorable.
Consequently it is skeletally generated, and according to
proposition~\ref{lattice}, $G$ has a $\sigma$-lattice consisting of
skeletal maps.

$(3)\Rightarrow (4)$. Suppose $G$ has a $\sigma$-lattice $\Psi_1$ of
skeletal maps. Then, by proposition~\ref{lattice} and
theorem~\ref{Ifav}, $G$ is $\mathrm{I}$-favorable. So, $G$ is
$\mathbb R$-factorizable group (see proposition ~\ref{IfavRfact}).
Hence, the family $\Psi_2$ of all continuous homomorphisms from $G$
onto second-countable groups is a $\sigma$-lattice.  Thus, by
proposition~\ref{propspectr}, the intersection $\Psi$ of these two
$\sigma$-lattices is a $\sigma$-lattice consisting of continuous
homomorphisms which are skeletal maps. Finally, because skeletal
homomorphism is nearly open~\cite[lemma 4.3.29]{ArTk}, $\Psi$ is a
$\sigma$-lattice of homomorphisms which are nearly open maps.

$(4)\Rightarrow (1)$ follows from proposition~\ref{sp-lat}.
\end{proof}

Proposition~\ref{proIfav} and theorem~\ref{Main} provide the
following properties of near-openly generated groups.

\begin{pro}\label{cor}
Let $G$ be a near-openly generated group. Then:
\begin{itemize}
\item every dense subgroup of group $G$ is also near-openly generated;
\item every topological group containing $G$ as a dense subgroup $($ in particular, the Raikov completion of $G$$)$ is near-openly generated;
\item every coset space of $G$ is skeletally generated $($in particular,
any quotient group of $G$ is near-openly generated$)$;
\item a product of near-openly generated groups is near-openly generated.
\end{itemize}
\end{pro}

Next proposition provides another interesting property of
near-openly generated groups.

\begin{pro}
Let $G$ be a near-openly generated group with $\displaystyle
G=a-\underleftarrow{\lim} S$, where $S$ is a factorizing almost
continuous $\sigma$-spectrum. Then spectrum $S$ has a cofinal
subspectrum which is an almost continuous $\sigma$-spectrum
consisting of groups and nearly open homomorphisms.
\end{pro}

\begin{proof}
Since there exists a $\sigma$-lattice of nearly open homomorphisms
on $G$, we apply proposition~\ref{lemmain} to find a cofinal
subspectrum of $S$ with the required properties.
\end{proof}

Our next aim is to establish the external characterization of
near-openly generated groups stated in theorem~\ref{Main1}.

\begin{pro}\label{trsp}
Let $G$ be a near-openly generated topological group of uncountable
weight $\tau$, $\lambda=\mathrm{cf}(\tau)$. Then there is an almost
continuous transfinite spectrum $\displaystyle S_G=\{G_\gamma,
p^{\delta}_\gamma, \gamma\leq\delta<\lambda\}$ of near-openly
generated topological groups $G_\gamma$ and nearly open
homomorphisms $p^{\delta}_\gamma$, such that $w(G_\gamma)<\tau$ for
each $\gamma$ and $G=\mathrm{a}-\displaystyle\lim_\leftarrow S_G$.
\end{pro}

\begin{proof}
According to theorem~\ref{Main} and~\cite[lemma 4.3.29]{ArTk}, it
suffices to find such a spectrum $\displaystyle S_G=\{G_\gamma,
p^{\delta}_\gamma, \gamma\leq\delta<\lambda\}$, where
$\lambda=\mathrm{cf}(\tau)$, of $\rm I$-favorable groups $G_\gamma$
and skeletal homomorphisms $p^{\delta}_\gamma$ such that
$w(G_\gamma)<\tau$ for each $\gamma$ and
$G=\mathrm{a}-\displaystyle\lim_\leftarrow S_G$. Since, by
proposition~\ref{IfavRfact}, $G$ is $\mathbb R$-factorizable, $G$ is
topologically isomorphic to a subgroup of the product of a family of
second-countable topological groups $G_\alpha$, $\alpha\in A$, with
$|A|=\tau$. For each $\alpha$ let $\overline{G}_\alpha$ be a
metrizable compactification of $G_\alpha$ and $A$ be the union of an
increasing transfinite family $\{A_\delta\}_{\delta<\lambda}$ with
$|A_\delta|<\tau$ for each $\delta<\lambda$. Then the closure
$\overline G$ of $G$ in the product $H=\prod_{\alpha\in
A}\overline{G}_\alpha$, being a compactification of $G$, is
$\mathrm{I}$-favorable (proposition~\ref{proIfav}). So, there is a
$\pi$-regular operator $\mathrm{e}:\Tee_{\overline G}\to\Tee_H$, see
theorem~\ref{Ifav}.

For every subset $B\subset A$ denote by $\pi_B:H\to
H_B=\prod_{\alpha\in B}\overline{G}_\alpha$ the projection,
$\widetilde G_B=\pi_B(\overline G)$ and $G_B=p_B(G)$, where
$p_B=\pi_B|_G$. We also fix a standard open base $\mathcal B_B$ for
$H_B$ of cardinality $\max\{|B|, \aleph_0\}$, $B\subset A$. For each
$U\in\mathcal B_B$ there is a countable set $k(U)\subset A$ with
$$\pi^{-1}_{k(U)}(\pi_{k(U)}(\overline{\mathrm{e}(\pi_B^{-1}(U)\cap
G)}))=\overline{\mathrm{e}(\pi_B^{-1}(U)\cap G)}.$$ This can be done
because each $\overline{\mathrm{e}(\pi_B^{-1}(U)\cap G)}$ is a
zero-set in $H$ and every continuous function on $H$ depends on
countably many coordinates. Following the proof of
proposition~\ref{lattice}, let $\mathcal A$ the family of all $\rm
e$-admissible subsets of $A$. Recall that the restriction
$\pi_B|_{\overline G}:\overline G\to\widetilde G_B$ is a skeletal
map for each $B\in\mathcal A$ and for every $B\subset A$ there is a
set $B_\infty\in\mathcal A$ containing $B$ with
$|B_\infty|=\max\{|B|, \aleph_0\}$, see the proof of
proposition~\ref{lattice}, $(1)\Rightarrow (2)$.

Next, we construct by transfinite induction an increasing family
$\{B_\delta\}_{\delta<\lambda}$ of $\mathrm{e}$-admissible sets
$B_\delta\subset A$ satisfying the following conditions:
\begin{itemize}
\item[(1)] $A_\delta\subset B_\delta$ and $|B_\delta|=|A_\delta|$;
\item[(2)] $B_\delta=\bigcup_{\gamma<\delta}B_\gamma$ provided $\delta$ is a limit ordinal.
\end{itemize}
So, $A$ is the union of an increasing transfinite family
$\{B_\delta\}_{\delta<\lambda}$ of $\mathrm{e}$-admissible sets.
Hence, the spectrum $\displaystyle S_G=\{G_\gamma,
p^{\delta}_\gamma, \gamma\leq\delta<\lambda\}$ is almost continuous
and $G=\mathrm{a}-\displaystyle\lim_\leftarrow S_G$, where
$G_\gamma=G_{B_\gamma}$ and $p^{\delta}_\gamma$ is the restriction
of the projection $\pi^{B_\delta}_{B_\gamma}:H_{B_\delta}\to
H_{B_\gamma}$ on $G_{B_\delta}$. Since each set $B_\gamma$ is
$\mathrm{e}$-admissible, the restrictions
$\pi_{B_\gamma}|\overline{G}:\overline G\to\widetilde G_{B_\gamma}$
are skeletal, so are the homomorphisms $p_{\gamma}:G\to G_{\gamma}$
since $G$ is a dense subset of $\overline{G}$. This implies that
each $G_{\gamma}$ is $\mathrm{I}$-favorable (as a skeletal image of
$G$), and the bonding homomorphisms $p^{\delta}_\gamma$ are also
skeletal.
\end{proof}

Now we can establish the external characterization of near-openly
generated groups.

\begin{thm}\label{Main1}
The following are equivalent for a topological group $G$:
\begin{itemize}
\item[{\rm (1)}] $G$ is near-openly generated;
\item[{\rm (2)}] $G$ is topologically isomorphic to a subgroup of a product of second-countable topological groups and any such an embedding is
$\pi$-regular;
\item[{\rm (3)}] $G$ is topologically isomorphic to a subgroup of a product of second-countable topological groups and any such an embedding is
regular.
\end{itemize}
\end{thm}

\begin{proof}
$(1)\Rightarrow (2)$. Suppose $G$ is a near-openly generated group.
Then, by theorem~\ref{Main}, $G$  is $\mathrm{I}$-favorable. So, $G$
is $\mathbb R$-factorizable (proposition~\ref{IfavRfact}).
Consequently, according to~\cite[proposition 8.1.3]{ArTk}, $G$ is
topologically isomorphic to a subgroup of a product of
second-countable topological groups. Finally, by
proposition~\ref{Ifav}, any embedding of $G$ into another space if
$\pi$-regular.

$(2)\Rightarrow (3)$. Suppose $G$ satisfies item $(2)$. Then $G$ is
$\pi$-regularly embedded in a product of a second-countable spaces
and, by corollary~\ref{Icor}, $G$ is $\mathrm{I}$-favorable. We are
going to show by transfinite induction with respect to $w(G)$ that
every subgroup of a product of second-countable topological groups,
topologically isomorphic to $G$, is regularly embedded in that
product. Suppose $w(G)$ is countable and let $G$ be a subgroup of a
product $\Pi=\prod_{\alpha\in\Gamma}G_\alpha$ of second-countable
groups. We can find a countable set $A\subset\Gamma$ and a countable
base $\mathcal B$ of $G$ such that $\pi_A|_G:G\to\pi_A(G)=G_A$ is an
isomorphism and $\pi_A^{-1}(\pi_A(V))\cap G=V$ for all $V\in\mathcal
B$. Since $\Pi_A=\prod_{\alpha\in A}G_\alpha$ is second-countable,
there is a regular operator $\mathrm{e}_A:\mathcal
T_{G_A}\to\mathcal T_{\Pi_A}$. Then $$\mathrm{e}:\mathcal
T_G\to\mathcal T_\Pi,\
\mathrm{e}(U)=\bigcup\{\pi_A^{-1}(\mathrm{e}_A(\pi_A(V))):V\in\mathcal
B\ \mbox{and}\ V\subset U\},$$ is a regular operator.

Suppose $w(G)=\tau$ is uncountable and our statement holds for every
$\mathrm{I}$-favorable group of weight $<\tau$. Let $G$ be  a
subgroup of a product $\Pi=\prod_{\alpha\in\Gamma}G_\alpha$ of
second-countable groups. For every sets $C\subset B\subset\Gamma$
let $\pi_B:\Pi\to\Pi_B=\prod_{\alpha\in B}G_\alpha$ and
$\pi^B_C:\Pi_B\to\Pi_C$ denote the corresponding projections. Take a
local base $\{U_\gamma:\gamma\in\Lambda\}$ at the neutral element in
$G$ of cardinality $|\Lambda|=\tau$ with each $U_\gamma$ being of
the form $U_\gamma=\pi_{B_\gamma}^{-1}(V_\gamma)\cap G$, where
$B_\gamma\subset\Gamma$ is finite and $V_\gamma$ is a neighborhood
of the neutral element in $\pi_{B_\gamma}(G)$. Then the set
$A=\bigcup_{\gamma\in\Lambda}B_\gamma$ is a subset of $\Gamma$ of
cardinality $\tau$ and $G$ is topologically isomorphic to
$\pi_A(G)$. We identify $G$ with the image $\pi_A(G)$. By
proposition~\ref{trsp}, $A$ can be covered by an increasing family
$A_\delta$, $\delta<\lambda=\mathrm{cf}(\tau)$, such that:
\begin{itemize}
\item $|A_\delta|<\tau$ for each $\delta$;
\item the spectrum $\displaystyle S_G=\{G_\delta, q^{\delta}_\eta, \eta\leq\delta<\lambda\}$ is almost continuous and consists of $\mathrm{I}$-favorable groups
and nearly open homomorphisms, where $G_\delta=\pi_{A_\delta}(G)$
and $q^{\delta}_\eta=\pi^{A_\delta}_{A_\eta}|_{G_\delta}$;
\item $G=\mathrm{a}-\displaystyle\lim_\leftarrow S_G$.
\end{itemize}
According to our assumption, for each $\delta$ there exists a
regular operator
$\mathrm{e}_\delta:\Tee_{G_\delta}\to\Tee_{\Pi_{A_\delta}}$. Then
proposition~\ref{regemb} implies the existence of a regular operator
$\mathrm{e}:\Tee_{G}\to\Tee_{\Pi_\lambda}$, where
$\Pi_\lambda=\prod_{\delta<\lambda}\Pi_{A_\delta}$. Denote by
$\pi:\Pi\to \Pi_\lambda$ the diagonal product of all homomorphisms
$\pi_{A_\delta}$ and for every open $U\subset G$ let
$\theta(U)=\pi^{-1}(\mathrm{e}(U))$. It is easily seen that
$\theta:\Tee_Q\to\Tee_\Pi$ is regular.

$(3)\Rightarrow (1)$. This implication follows from
corollary~\ref{Icor} and theorem~\ref{Main}.
\end{proof}

\begin{cor} Let $\displaystyle S_G=\{G_\gamma, p^{\delta}_\gamma,
\gamma\leq\delta<\lambda\}$ be a transfinite almost continuous
spectrum of near-openly generated groups $G_\gamma$ and nearly open
homomorphisms $p^{\delta}_\gamma$ such that
$G=\mathrm{a}-\displaystyle\lim_\leftarrow S_G$. Then $G$ is
near-openly generated group.
\end{cor}

\begin{proof}
By theorem~\ref{Main1}, each $G_\gamma$ is regularly embedded in a
product $H_\gamma$ of second-countable groups. So, according to
\cite[proposition 4.3]{v}, $G$ is regularly embedded in the product
$\prod_{\gamma<\delta}H_\gamma$. Finally, theorem~\ref{Main1}
completes the proof.
\end{proof}

Next theorem provides local characterization of near-openly
generated groups.

\begin{thm}\label{local}
An $\omega$-narrow group $G$ is near-openly generated iff $G$ has an
$\mathrm{I}$-favorable neighborhood of the identity.
\end{thm}

\begin{proof}
If $G$ is near-openly generated, then it is $\mathrm{I}$-favorable
and, by proposition~\ref{proIfav}, every open subset of $G$ is
$\mathrm{I}$-favorable.

Suppose $G$ is $\omega$-narrow and let $V$ be an
$\mathrm{I}$-favorable neighborhood of the identity.

First, let show that $G$ is $\mathbb R$-factorizable. As an
$\omega$-narrow group, $G$ is isomorphically embedded into a product
$\Pi=\prod_{\alpha\in A}G_\alpha$ of second-countable groups. Taking
an open subset of $V$ if necessarily, we can assume that $V=G\cap O$
where $O$ is a co-zero set in $\Pi$. By theorem~\ref{Ifav}, $V$ is
$\pi$-regularly embedded in $O$ and, according to
proposition~\ref{pi-z}, $V$ is $z$-embedded in $O$. Therefore, for
every co-zero set $U$ in $V$ there is a co-zero set $W\subset O$
with $U=V\cap W$. Thus, $U=G\cap W$. Since $G$ is $\omega$-narrow,
there is a sequence $\{g_i\}\subset G$ such that $\{g_iV:
i\in\mathbb N\}$ covers $G$. Each set $g_iV$ is
$\mathrm{I}$-favorable being homeomorphic to $V$. Moreover,
$g_iO\cap G=g_iV$ and each $g_iV$ is $z$-embedded in $g_iO$. As
above, we can show that for every co-zero set $U_i\subset g_iV$
there is a co-zero set $W_i\subset g_iO$ with  $U_i=G\cap W_i$,
$i\geq 1$. This yields that $G$ is $z$-embedded in $\Pi$. Indeed, if
$U$ is a co-zero set in $G$, then each $U_i=U\cap g_iV$ is a co-zero
set in $g_iV$. Hence, for every $i$ there is a co-zero set
$W_i\subset g_iO$ such that $U_i=G\cap W_i$. Finally, the set
$W=\bigcup_{i=1}^\infty W_i$ is a co-zero set in $\Pi$ because so is
every $W_i$, and $W\cap G=U$. Thus, $G$ is $z$-embedded in $\Pi$.
Hence, $G$ is an $\mathbb R$-factorizable group (see \cite[theorem
8.2.6]{ArTk}).

We can complete the proof. Since $V$ is $\mathrm{I}$-favorable,
there exists a $\sigma$-lattice $L$ of skeletal maps on $V$, see
proposition~\ref{lattice}. On the other hand, since $G$ is $\mathbb
R$-factorizable, all continuous homomorphisms on $G$ having
second-countable images form a factorizing almost continuous
$\sigma$-spectrum $S_G$ with
$G=\mathrm{a}-\displaystyle\lim_\leftarrow S_G$. Then, by
proposition~\ref{embfact}, the restriction $S_G|_V$ is a factorizing
almost continuous $\sigma$-spectrum and
$V=\mathrm{a}-\displaystyle\lim_\leftarrow S_G|_V$. So, according to
proposition~\ref{lemmain}, there exists a cofinal almost continuous
$\sigma$-subspectrum $S$ of $S_G|_V$ consisting of skeletal maps.
Observe that $S$ is the restriction of an almost continuous
$\sigma$-spectrum $\widetilde S=\{G_\alpha, p^\beta_\alpha, A\}$, a
subspectrum of $S_G$. Since $S$ is cofinal in $S_G|_V$, $\widetilde
S$ is also cofinal in $S_G$. Hence, $\widetilde S$ is factorizing
(because so is $S_G$) and
$G=\mathrm{a}-\displaystyle\lim_\leftarrow\widetilde S$. Moreover,
because every locally skeletal homomorphism is skeletal, all
projections $p_\alpha:G\to G_\alpha$ are nearly open homomorphisms,
see~\cite[lemma 4.3.29]{ArTk}.  Therefore, $G$ is a near-openly
generated group.
\end{proof}

\subsection{Examples and questions}

\begin{ex}
1) Any compact group is near-openly generated and has a
$\sigma$-lattice of open homomorphisms onto metrizable compact
groups.

2) The function space $C_p(X)$ equipped with the pointwise
convergence topology is near-openly generated as a dense subgroup of
the product $\mathbb R^X$.

3) Any topological group $G$ homeomorphic to $\mathrm{I}$-favorable
space with a $\sigma$-lattice of open maps is near-openly generated
and has a $\sigma$-lattice of open homomorphisms.

\begin{proof}
Indeed, let $\Psi_{op}$ be a $\sigma$-lattice on $G$ of open maps.
Then, by theorem~\ref{Main}, $G$ is near-openly generated and has a
$\sigma$-lattice $\Psi_{no}$ of nearly open homomorphisms. According
to proposition~\ref{propspectr}, the family
$\Psi=\Psi_{no}\cap\Psi_{op}$ is also a $\sigma$-lattice on $G$ and
consists of open homomorphisms.
\end{proof}

4) Let a group $G$ be homeomorphic to a product $\Pi=\Pi\{G_t: t\in
T\}$ of near-openly generated groups such that  every $G_t$ has a
$\sigma$-lattice of open homomorphisms. Then $G$ is near-openly
generated and has a $\sigma$-lattice of open homomorphisms.

\begin{proof}
Let $\Psi_t$ be a $\sigma$-lattice for $G_t$, $t\in T$, consisting
of open homomorphisms. We can assume that each $\Psi_t$ does not
contain a constant homomorphism. For every countable set $A\subset
T$ we consider the family $\mathcal M_A$ of all maps on the
subproduct $\Pi_A=\prod_{t\in A}G_t$ having the form $\prod_{t\in
A}f_t$ with $f_t\in\Psi_t$ for each $t\in A$. Denote by $\mathcal
L_A$ the family $\{f\circ\pi_A:f\in\mathcal M_A\}$, where $\pi_A$ is
the projection onto $\Pi_A$, and let $\Psi_G$ consist of all
families $\mathcal L_A$, $A$ being a countable subset of $T$. One
can show that if $f\circ\pi_A\prec g\circ\pi_B$ for some
$f=\prod_{t\in A}f_t\in\mathcal L_A$ and  $g=\prod_{t\in
B}g_t\in\mathcal L_B$, then $B\subset A$ and $f_t\prec g_t$ for all
$t\in B$. This easily implies that $\Psi_G$ satisfies condition
$(L1)$ from the definition of a $\sigma$-lattice. It is also true
that for any two maps $f\circ\pi_A\in\mathcal L_A$ and
$g\circ\pi_B\in\mathcal L_B$ there is $h\circ\pi_{A\cup
B}\in\mathcal L_{A\cup B}$ such that $h\circ\pi_{A\cup B}\prec
f\circ\pi_A$ and $h\circ\pi_{A\cup B}\prec g\circ\pi_B$. So,
$\Psi_G$ is directed with respect to the relation $\prec$. Moreover,
$\Psi_G$ consists of open homomorphisms and generates the topology
of $\Pi$. Therefore, $\Psi_G$ generates an almost continuous
$\sigma$-spectrum $S_G$ with
$\Pi=\mathrm{a}-\displaystyle\lim_\leftarrow S_G$, see
proposition~\ref{sp-lat}. Since $\Pi$ is $\mathrm{I}$-favorable (as
a product of $\mathrm{I}$-favorable spaces), by
proposition~\ref{IfavRfact}, $\Pi$ is $\mathbb R$-factorizable.
Finally, theorem~\ref{factsig} yields that $S_G$ is a factorizable
spectrum. Consequently, $\Psi_G$ is a $\sigma$-lattice of open
homomorphisms, and we apply example 1.3 to conclude that $G$ is
near-openly generated and has a $\sigma$-lattice of open
homomorphisms.
\end{proof}

5)\ \textit{Almost connected pro-Lie groups.} Pro-Lie groups were
introduced in~\cite{hm} as the groups which are projective limits of
Lie groups. Any pro-Lie group $G$ such that the quotient group
$G/G_0$ is compact, where $G_0$ is the connected component of $G$,
is  called {\it almost connected pro-Lie group}, see~\cite{hm1}. By
\cite[corollary 8.9]{hm1} any almost connected pro-Lie group is
homeomorphic to a product of a compact topological group and a power
of $\mathbb R$. Thus, $G$ is near-openly generated and has a
$\sigma$-lattice of open homomorphisms, see example 1.4.

6) Let $K$ be a compact invariant subgroup of $G$. If the quotient
group $H=G/K$ is an $\mathrm{I}$-favorable space with a
$\sigma$-lattice of open maps, then $G$ is a near-openly generated
group with a $\sigma$-lattice of open homomorphisms.

\begin{proof}
We use the construction from~\cite[theorem 3.11]{TkL}. By example
1.3, $H$ has a $\sigma$-lattice $\Psi_H$ of open homomorphisms. Let
${\mathcal A}_H=\{\mathrm{Ker}(\varphi):\varphi\in\Psi_H\}$.
Consider also the family $\mathcal A_G$ of all invariant subgroups
$N$ of $G$ such that the quotient group $G/N$ is a second-countable
space and $\pi (N)\in{\mathcal A}_H$, where $\pi:G\to H$ is the
quotient map. It follows from~\cite[lemma 3.3]{TkL} and the proof
of~\cite[theorem 3.11]{TkL} that the family $\Psi_G$ of all quotient
maps correspondent to ${\mathcal A}_G$ is a $\sigma$-lattice for $G$
consisting of open homomorphisms. Let us note that the requirement
in~\cite[theorem 3.11]{TkL} to have a strong $\sigma$-lattice on $H$
is not necessary in our case.
\end{proof}

\begin{cor}
Let $K$ be a compact invariant subgroup of $G$. If the quotient
group $H=G/K$ is homeomorphic to an almost connected pro-Lie group,
then $G$ is a near-openly generated group with a $\sigma$-lattice of
open homomorphisms.
\end{cor}

7)\ \textit{Lindel\" off almost metrizable groups.} A group is
almost metrizable~\cite{pa} if it contains a compact subset of
countable character (almost metrizable groups coincides with {\em
feathered groups}~\cite{ArTk}). In case an almost metrizable group
$G$ is lindel\" off it has a compact invariant subgroup $K$ of
countable character such that the quotient group $G/K$ is
second-countable. So, by Example 1.6, $G$ is a near-openly generated
group with a $\sigma$-lattice of open homomorphisms.

8) Any near-openly generated group has countable cellularity. Since
there are Lindel\" off groups of uncountable cellularity and any
Lindel\" off group is $\mathbb R$-factorizable, there is an $\mathbb
R$-factorizable group which is not near-openly generated.
\end{ex}

\begin{qu}
Let $K$ be a compact invariant subgroup of a topological group $G$
such that the quotient group $H=G/K$ is near-openly generated. Is it
true that $G$ is near-openly generated?
\end{qu}

\begin{qu}
Let $G$ be an almost limit of an almost $\sigma$-continuous spectrum
$($see definition in next section, just before proposition
$\ref{lemmain}$$)$ $S_G$ of near-openly generated groups and nearly
open homomorphisms. Is $G$ a near-openly generated group?
\end{qu}

\section{Appendix}
\subsection{Some properties of spectra}
In this section statements used in the previous sections are
provided.

By a {\em cofinal subspectrum} of a spectrum $\displaystyle
S=\{X_\alpha, p^{\beta}_\alpha, A\}$ we mean a spectrum
$\displaystyle S'=\{X_\alpha, p^{\beta}_\alpha, A'\}$ such that the
index set $A'$ is cofinal in $A$, i.e., for every $\alpha\in A$
there is $\gamma\in A'$ with $\alpha<\gamma$. Clearly, every cofinal
subspectrum of a factorizing spectrum is also factorizing. Below, if
$\Psi$ is a family of maps on a space $X$ we say that {\em $\Psi$ is
directed} if it is directed with respect to the order on $\Psi$
defined by $f<g$ iff $g\prec f$. For any two maps $f,g\in\Psi$ with
$f<g$ there is a unique map $p^g_f:g(X)\to f(X)$ such that
$f=p^g_f\circ g$. The proof of the next proposition needs only
comparison of correspondent definitions.

\begin{pro}\label{sp-lat}
Let $\Psi$ be a directed family of maps on $X$ onto second-countable
spaces such that $\Psi$ generates the topology of $X$ and satisfies
condition $(L1)$. Then  $S_\Psi=\{f(X),p^g_f, \Psi\}$ is an almost
continuous $\sigma$-spectrum with $X=\rm
a-\displaystyle\lim_\leftarrow S_\Psi$. If, in addition, $\Psi$ is a
$\sigma$-lattice, $S_\Psi$ is a factorizing. Moreover, every
$\sigma$-sublattice of $\Psi$ generates a cofinal almost continuous
$\sigma$-subspectrum $S'$ of $S_\Psi$ with $X=\rm
a-\displaystyle\lim_\leftarrow S'$.
\end{pro}

Besides almost continuous $\sigma$-spectra, we also consider spectra
$\displaystyle S=\{X_\alpha, p^{\beta}_\alpha, A\}$ such that all
bonding maps $p^{\beta}_\alpha$ are surjective, the directed set $A$
is $\sigma$-complete and for any  countable chain
$\{\alpha_n\}_{n\geq 1}\subset A$ with
$\beta=\sup\{\alpha_n\}_{n\geq 1}$ the space $X_\beta$ is a (dense)
subset of
$\displaystyle\lim_\leftarrow\{X_{\alpha_n},p^{\alpha_{n+1}}_{\alpha_n},
n\geq 1\}$. Any such a spectrum is said to be {\em almost
$\sigma$-continuous}. In other words, almost $\sigma$-continuous
spectra satisfy the conditions determining almost continuous
$\sigma$-spectra but the spaces $X_\alpha$ are not required to be
second-countable.

\begin{pro}\label{lemmain} Let $f$ be a homeomorphism of $X$
onto $Y$, $S_X=\{X_\alpha, p^\alpha_\beta, A_X\}$ be a factorizing
almost continuous $\sigma$-spectrum with
$X=a-\displaystyle\lim_\leftarrow S_X$ and $\Psi_Y$ be a
$\sigma$-lattice on $Y$. Then $f$ is induced by an isomorphism of
cofinal subspectra $S_X'$ of $S_X$ and $S_Y'$ of $S_Y$, where $S'_X$
is an almost continuous $\sigma$-subspectra of $S_X$ and $S_Y$ is
the factorizing almost continuous $\sigma$-spectrum generated by
$\Psi_Y$.
\end{pro}

\begin{proof}
Without loss of generality, we can assume that $X=Y$ and $f$ is the
identity. So, $\Psi=\Psi_Y$ is a $\sigma$-lattice on $X$. Denote by
$S_\Psi$ the factorizing almost continuous $\sigma$-spectrum
generated by $\Psi$, see proposition~\ref{sp-lat}. Hence, our proof
is reduced to show that there is a cofinal almost continuous
$\sigma$-subspectrum of $S_X$ which is a cofinal  subspectrum of
$S_\Psi$. To this end, let $\mathcal P=\{p_\alpha:\alpha\in A_X\}$.
Define a partial order on $\mathcal P$ by $p_\alpha<_{\mathcal
P}p_\beta$ iff $\alpha<\beta$. Obviously, $p_\alpha<_{\mathcal
P}p_\beta$ implies $p_\beta\prec p_\alpha$. The relation
$\varphi_1<_\Psi\varphi_2$ iff $\varphi_2\prec\varphi_1$ defines
also a partial order on the set $\Psi$.

\smallskip

\noindent{\it Claim 2.} $\Psi\cap\mathcal P$ satisfies condition
(L2) and it is cofinal in both $\Psi$ and $\mathcal P$.

\smallskip

Indeed, let $h$ be a continuous function on $X$. Using that $\Psi$
satisfies condition $(L2)$ and $S_X$ is factorizing, we can
construct by induction two sequences
$\{p_{\alpha_n}\}\subset\mathcal P$ and $\{\varphi_n\}\subset\Psi$
such that $\varphi_{n+1}\prec p_{\alpha_n}\prec\varphi_n\prec h$ and
$\alpha_n<\alpha_{n+1}$ for all $n$.  Then, because $\Psi$ is a
$\sigma$-lattice and $S_X$ is almost $\sigma$-continuous, we have
$\varphi=\triangle_{n\geq 1}\varphi_n\in\Psi$ and
$p_\alpha\in\mathcal P$, where $\alpha=\sup\alpha_n\in A_X$.
Moreover, our construction yields that $\varphi(X)$ is homeomorphic
to $X_\alpha$. So, $\varphi=p_\alpha$ belongs to $\Psi\cap\mathcal
P$ and $\varphi\prec h$. Similarly, taking $h\in\Psi$ (resp.,
$h\in\mathcal P$), we can find $\varphi\in\Psi\cap\mathcal P$ such
that $h<_\Psi\varphi$ (resp., $h<_\mathcal P\varphi$). This
completes the proof of the claim.

\medskip

Below we identify every element $\phi\in\Psi\cap\mathcal P$ with a
couple $(\varphi,p_\alpha)$ such that $\varphi\in\Psi$, $\alpha\in
A_X$ and the spaces $\varphi(X)$ and $X_\alpha=p_\alpha(X)$ are
homeomorphic. We introduce a partial order on $\Psi\cap\mathcal P$:
$(\varphi_1,p_{\alpha_1})<*(\varphi_2,p_{\alpha_2})$ iff
$\alpha_1<\alpha_2$. Obviously, if $\phi_1<*\phi_2$ for some
$\phi_1, \phi_2\in\Psi\cap\mathcal P$, then
$\varphi_2\prec\varphi_1$. Using again that $\Psi$ is a
$\sigma$-lattice and the spectrum $S_X$ is almost
$\sigma$-continuous, one can show that the set $\Psi\cap\mathcal P$
is $\sigma$-complete with respect to the order $<*$, i.e. if
$\{(\varphi_n,p_{\alpha_n})\}$ is a sequence in $\Psi\cap\mathcal P$
such that
$(\varphi_n,p_{\alpha_n})<*(\varphi_{n+1},p_{\alpha_{n+1}})$ for
each $n$, then $\{(\varphi_n,p_{\alpha_n})\}$ has a supremum in
$\Psi\cap\mathcal P$ and
$\sup\{(\varphi_n,p_{\alpha_n})\}=(\varphi,p_\alpha)$, where
$\varphi=\triangle_{n\geq 1}\varphi_n$ and $\alpha=\sup\alpha_n\in
A_X$. This implies that the subspectrum $S_{\Psi\cap\mathcal P}$ of
$S_X$ is almost $\sigma$-continuous. Finally, by claim 2,
$S_{\Psi\cap\mathcal P}$ is also factorizing and it is a cofinal
subspectrum of $S_X$ and $S_\Psi$.
\end{proof}

Next proposition is an analogue of E.~Shchepin's theorem~\cite{sc1}
on intersection of lattices. It follows from
proposition~\ref{lemmain}.

\begin{pro}\label{propspectr}
If $\Psi_1$ and $\Psi_2$ are two $\sigma$-lattices on $X$, then
$\Psi_1\cap\Psi_2$ is a $\sigma$-lattice on $X$.
\end{pro}

\begin{pro}\label{FS} If $X=\mathrm{a}-\displaystyle\lim_\leftarrow S$, where $S$ is a factorizing almost continuous $\sigma$-spectrum,
then $\displaystyle\lim_\leftarrow S$ is the realcompactification of
$X$.
\end{pro}

\begin{proof}
Because $S$ is factorizing, $\displaystyle\lim_\leftarrow S$ is a
realcompact space in which $X$ is $C$-embedded. Hence, $v
X=\displaystyle\lim_\leftarrow S$.
\end{proof}

\begin{pro}\label{embfact} Let $\displaystyle S_X=\{X_\alpha, p^{\beta}_\alpha, A\}$ be a factorizing almost $\sigma$-continuous spectrum
and $X=\mathrm{a}-\displaystyle\lim_\leftarrow S$. If $Y\subset X$
is a co-zero set, then $\displaystyle S_Y=\{Y_\alpha=p_\alpha(Y),
p^{\beta}_\alpha|_{Y_\beta}, A\}$ is a factorizing almost
$\sigma$-continuous spectrum.
\end{pro}

\begin{proof}
Obviously, $S_Y$ is almost $\sigma$-continuous, so we need to show
it is also factorizing. Suppose $Y=g^{-1}((0,1])$ for some
continuous function $g$ on $X$. Since $S_X$ is factorizing, there
exists $\gamma\in A$ and a continuous function $g_\gamma:
X_\gamma\to\mathbb R$ with $g=g_\gamma\circ p_\gamma$.

Let $f: Y\to\mathbb R$ be a continuous function. Put
$$f_n(x)=\left\{\begin{array}{cc}
f'_n(x)\cdot g_n(x), & x\in Y\\
0, & x\not\in Y,\\
\end{array}\right.$$

where
$$f'_n(y)=\left\{\begin{array}{cc}
\min\{f(y),n\}, & f(y)\geq 0\\
\max\{f(y),-n\}, & f(y)< 0,\\
\end{array}
\right.$$ and $g_n(x)=\min\{n\cdot g(x), 1\}$, $x\in X$.

The sequence $\{f_n\}$ consists of continuous functions on $X$ such
that $\lim_{n\to\infty}f_n(y)=f(y)$ for any $y\in Y$. Because
spectrum $S_X$ is factorizing, we can construct by induction a
sequence $\{\alpha_n\}\subset A$ and continuous functions  $h_n:
X_{\alpha_n}\to\mathbb R$ such that
$\alpha_{n+1}\geq\alpha_n>\gamma$ and $f_n=h_n\circ p_{\alpha_n}$.
Let $\beta=\sup\{\alpha_n: n\in\mathbb N\}$. Then for any points $y,
y'\in Y$ such that $p_\beta(y)=p_\beta(y')$ we have
$$\lim_{n\to\infty}f_n(y)=\lim_{n\to\infty}h_n\circ
p_{\alpha_n}(y)=\lim_{n\to\infty}h_n\circ
p_{\alpha_n}(y')=\lim_{n\to\infty}f_n(y'),$$ which implies
$f(y)=f(y')$. Therefore, there is a map $h_\beta: Y_\beta\to\mathbb
R$ such that $f=h_\beta\circ (p_{\beta}|_{Y})$. It remains to show
$h_\beta$ is continuous. Suppose $p_\beta (y)=z$ for some $z\in
Y_\beta$ and $y\in Y$. Then there is $k$ such that $g_n(y)=1$ and
$f_n(y)=f'_n(y)=f(y)\in (-k,k)$ for all $n\geq k$. Fix $m>k$ and
$\varepsilon>0$, and let $y_m=p_{\alpha_m}(y)$. So,
$f_m(y)=h_m(p_{\alpha_m}(y))\in (-k,k)$ and $y\in g^{-1}((\frac1{m},
1])$. Since $\alpha_m>\gamma$, there is a map
$p^{\alpha_m}_\gamma:X_{\alpha_m}\to X_\gamma$ with
$p_\gamma=p^{\alpha_m}_\gamma\circ p_{\alpha_m}$. Then
$g_\gamma\big(p^{\alpha_m}_\gamma(y_m)\big)=g(y)\in (\frac1{m}, 1]$.
Consequently, there is a neighborhood $W\subset X_{\alpha_m}$ of
$y_m$ such that $W\subset (g_\gamma\circ
p^{\alpha_m}_\gamma)^{-1}((\frac1{m}, 1])\cap h_m^{-1}((-k,k))$ and
$|h_m(t)-h_m(y_m)|<\varepsilon$ for every $t\in W$. Since $g_m(x)=1$
and $f_m(x)\in (-k,k)$ for all $x\in p_{\alpha_m}^{-1}(W)$,
$p_{\alpha_m}^{-1}(W)\subset f^{-1}((-m,m))$. Therefore, if
$t=p_{\alpha_m}(x)$, where $x\in p_{\alpha_m}^{-1}(W)$, we have
$$\varepsilon>|h_{m}(t)-h_{m}\circ p_{\alpha_m}(y)|=|h_{m}\circ p_{\alpha_m}(x)-h_{m}\circ
p_{\alpha_m}(y)|=$$
$$=|f_m(x)-f_m(y)|=|f(x)-f(y)|.$$
The last inequality implies $$|h_\beta(z')-h_\beta(z)|=|h_\beta\circ
p_\beta(x)-h_\beta\circ p_\beta(y)|=|f(x)-f(y)|<\varepsilon$$ for
all $z'=p_\beta(x)\in (p^\beta_{\alpha_m})^{-1}(W)$. Thus, $h_\beta$
is continuous.
\end{proof}

\subsection{Spectral representations of $\mathbb R$-factorizable
groups}

If $G$ is an $\mathbb R$-factorizable group, then the family of all
continuous homomorphisms from $G$ onto second-countable groups is a
$\sigma$-lattice. So, proposition~\ref{lemmain} yields that if $G$
is an $\mathbb R$-factorizable group, then every factorizing almost
continuous $\sigma$-spectrum $S_G$ with $\displaystyle G=\rm
a-\underleftarrow{\lim} S_G$ has a cofinal almost continuous
$\sigma$-subspectrum consisting of groups and homomorphisms.

Next theorem implies that if $G$ is an $\mathbb R$-factorizable
group, then any almost $\sigma$-continuous spectrum $S$ of groups
and homomorphisms with $\displaystyle G=\rm a-\underleftarrow{\lim}
S$ is factorizing.

\begin{thm}\label{factsig}
Let $G$ be an $\mathbb R$-factorizable group and $\displaystyle
G=\rm a-\underleftarrow{\lim} S_G$, where $S_G$ is an almost
$\sigma$-continuous spectrum consisting of topological groups and
homomorphisms. Then $S_G$ is factorizing.
\end{thm}

\begin{proof} Since $G$ is $\mathbb R$-factorizable, it has the property $\omega$-$U$~\cite{XL}: for any continuous function
$f:G\to\mathbb R$ and for any $\varepsilon>0$ there exists a
countable family $\mathcal U(f,\varepsilon)=\{U_n\in N_G(e):
n\in\mathbb N\}$ such that for every $x\in G$, there exists $U_n$
such that $|f(x)-f(y)|<\varepsilon$ whenever $y\in U_nx$. Then
$\mathcal W(f)=\bigcup_{k\geq 1}\mathcal U(f,1/k)$ is also a
countable family. Without loss of generality we may assume that each
family $\mathcal W(f)=\{W_k:k\geq 1\}$ satisfies the additional
condition $W_k=W_k^{-1}$ and $W_{k+1}^2\subset W_k$.

Suppose $\displaystyle S_G=\{G_\alpha, p^{\beta}_\alpha, A\}$ and
$f$ is a continuous function on $G$. We construct an increasing
sequence $\{\alpha_k\}\subset A$ and neighborhoods $V_{\alpha_k}$ of
the identities $\rm e_{\alpha_k}$ of $G_{\alpha_k}$ respectively,
such that $p_{\alpha_k}^{-1}(V_{\alpha_k})\subset W_k$, $W_k\in
\mathcal W(f)$. Let $\alpha=\sup\{\alpha_k\}$. Since
$p_{\alpha}^{-1}(\rm e_\alpha)\subset\bigcap_{k\geq 1}W_k$, we have
$p_\alpha(x)=p_\alpha(y)$ iff $y\in\bigcap\{W_kx: k\in\mathbb N\}$.
Hence, $p_\alpha(x)=p_\alpha(y)$ implies $f(x)=f(y)$. Therefore, we
can define a map $f_\alpha: G_\alpha\to\mathbb R$ such that
$f_\alpha\circ p_\alpha=f$.

To finish the proof it remains to show that $f_\alpha$ is
continuous. To this end, let $x\in G_\alpha$ and $x'\in
p_\alpha^{-1}(x)$. For any $\varepsilon>0$ there exists
$W_k\in\mathcal W(f)$ such that $|f(x')-f(y)|<\varepsilon$ for every
$y\in W_kx'$. Then $(p_{\alpha_k}^\alpha)^{-1}(V_{\alpha_k})x$ is a
neighborhood  of $x$ in $G_\alpha$ and
$|f_\alpha(x)-f_\alpha(y)|<\varepsilon$ for every $y\in
(p_{\alpha_k}^\alpha)^{-1}(V_{\alpha_k})x$. Hence, $f_\alpha$ is
continuous at $x$.
\end{proof}

\begin{cor}\label{RI} If an $\mathbb R$-factorizable group $G$ has
a system $L$ of $($open, resp., nearly open$)$ homomorphisms onto
second-countable topological groups such that $L$ satisfies
condition $(L1)$ and generates the topology of $G$, then $L$ is a
$\sigma$-lattice of $($open, resp., nearly open$)$ homomorphisms.
\end{cor}

If an $\mathbb R$-factorizable group $G$ is a subgroup of a product
of second-countable groups, then the images of the restrictions to
$G$ of all projections into countable subproducts form an almost
continuous $\sigma$-spectrum. Therefore we have the following
corollary.

\begin{cor}  A topological group $G$ is $\mathbb R$-factorizable iff for any topological isomorphism of $G$ into a product
$\Pi=\prod_{\alpha\in A}G_\alpha$ of second-countable groups the
family $\{p_B: B\subset A\ \mbox{is countable}\}$, where $p_B:
G\to\prod_{\alpha\in B}G_\alpha$ denotes the projection, is a
$\sigma$-lattice on $G$.
\end{cor}

It is worth noting that any $\omega$-narrow not $\mathbb
R$-factorizable group is an almost limit of an almost continuous
$\sigma$-spectrum of second-countable groups and homomorphisms, but
no such a spectrum is factorizing. While $\mathbb R$-factorizable
groups can be characterized as those groups which are almost limits
of almost continuous $\sigma$-spectra of second-countable groups and
homomorphisms, and any such spectrum is factorizing.

\begin{cor}\label{vR} If $G$ is an $\mathbb R$-factorizable group and $\displaystyle G=\rm a-\underleftarrow{\lim} S_G$, where $S_G$ is an almost continuous $\sigma$-spectrum consisting of groups and
homomorphisms, then $\underleftarrow{\lim} S_G$ is homeomorphic to
the realcompactification of $G$.
\end{cor}

\begin{proof}
By theorem~\ref{factsig}, $S_G$ is a factorizing almost continuous
$\sigma$-spectrum. The rest follows from proposition~\ref{FS}.
\end{proof}

Corollary~\ref{vR} improves \cite[theorem 2.22]{k2017} stating that
if $G$ is an $\mathbb R$-factorizable group, then $vG$ is the limit
of a factorizing almost continuous $\sigma$-spectrum of
second-countable groups and homomorphisms.

\smallskip

\textbf{Acknowledgments.} The authors would like to express their
gratitude to M.~Tkachenko and A.~Kucharski for several discussions.

\bibliographystyle{amsplain}

\end{document}